\documentclass[11pt]{article}
\usepackage{geometry,amsthm,amssymb,amsmath,enumerate,float,cite, algorithm2e,verbatim}
\newtheorem{theorem}{Theorem}
\newtheorem{conjecture}{Conjecture}
\newtheorem{lemma}{Lemma}
\newtheorem{claim}{Claim}
\newtheorem{corollary}{Corollary}

\usepackage[utf8]{inputenc}
\usepackage[english]{babel}
\usepackage{tikz,tkz-graph}
\usepackage{amsfonts, amsopn,amsmath, amsthm}

\title{Linear programming based approximation for unweighted\\ 
induced matchings  ---  breaking the $\Delta$ barrier}
\author{Julien Baste \and Maximilian F\"{u}rst  \and Dieter Rautenbach}
\date{}

\begin{document}

\maketitle

{\small 
\begin{center}
Institute of Optimization and Operations Research, Ulm University, Germany\\
\texttt{\{julien.baste,maximilian.fuerst,dieter.rautenbach\}@uni-ulm.de}
\end{center}
}

\begin{abstract}
A matching in a graph is induced if no two of its edges are joined by an edge,
and finding a large induced matching is a very hard problem.
Lin et al. 
(Approximating weighted induced matchings, Discrete Applied Mathematics 243 (2018) 304-310)
provide an approximation algorithm 
with ratio $\Delta$
for the weighted version of the induced matching problem
on graphs of maximum degree $\Delta$.
Their approach is based on an integer linear programming formulation
whose integrality gap is at least $\Delta-1$,
that is, their approach offers only little room for improvement in the weighted case.
For the unweighted case though, 
we conjecture that the integrality gap is at most $\frac{5}{8}\Delta+O(1)$,
and that also the approximation ratio can be improved at least to this value.
We provide primal-dual approximation algorithms with ratios 
$(1-\epsilon) \Delta + \frac{1}{2}$ for general $\Delta$
with $\epsilon \approx 0.02005$,
and $\frac{7}{3}$ for $\Delta=3$.
Furthermore, we prove a best-possible bound 
on the fractional induced matching number
in terms of the order and the maximum degree.
\end{abstract}
{\small 
\begin{tabular}{lp{13cm}}
{\bf Keywords:} & Induced matching; strong matching; approximation algorithm; linear programming; primal-dual approximation algorithm\\
{\bf MSC 2010:} & 05C70
\end{tabular}
}

\newpage

\section{Introduction}
We consider finite, simple, and undirected graphs, and use standard terminology.
A set $M$ of edges of a graph $G$ is an {\it induced matching} in $G$ 
if no two edges in $M$ are adjacent or joined by an edge,
that is, $M$ is an independent set of the square $L^2(G)$ 
of the line graph $L(G)$ of $G$.
The {\it induced matching number} $\nu_s(G)$ of $G$ is the maximum cardinality of an induced matching in $G$.

The problem to find a maximum induced matching in a given graph
does not allow an efficient approximation algorithm with approximation factor $n^{1/2-\epsilon}$ for some positive $\epsilon$, unless $P=NP$ \cite{orfigozv},
and it is APX-complete for $\Delta$-regular bipartite graphs \cite{dadelo}.
Several efficient approximation algorithms have been proposed for $\Delta$-regular graphs:
Duckworth, Manlove, and Zito \cite{dumazi,zi} showed that a simple greedy strategy has approximation ratio $\Delta-O(1)$.
Combining the greedy strategy with local search, 
Gotthilf and Lewenstein \cite{gole} improved this to $0.75\Delta+0.15$.
For $\Delta$-regular $\{ C_3,C_5\}$-free graphs, 
Rautenbach \cite{ra} showed that the algorithm from \cite{gole} has approximation ratio $0.708\bar{3}\Delta+0.425$. 
Finally, for $\Delta=3$, that is, for cubic graphs, Joos, Rautenbach, and Sasse \cite{jorasa} 
described an efficient algorithm with approximation ratio $\frac{9}{5}$.
All these approximation ratios for $\Delta$-regular graphs rely on the simple upper bound 
\begin{eqnarray}
\nu_s(G)\leq \frac{m(G)}{2\Delta-1},\label{e0}
\end{eqnarray}
which fails for not necessarily regular graphs of maximum degree $\Delta$.

Only very recently, Lin, Mestre, and Vasiliev \cite{limeva}
improved the straightforward approximation ratio of $2(\Delta-1)$ of the greedy algorithm \cite{zi}
applied to a graph of maximum degree $\Delta$.
Their approach relies on linear programming and a local ratio technique.
They actually consider the weighted version of the problem, 
and provide an efficient algorithm with approximation ratio $\Delta$.
As they show that the integrality gap of their integer linear programming formulation 
of the weighted induced matching problem is at least $\Delta-1$,
there is not much room for improvement 
of the approximation ratio using their approach.

In order to phrase the integer linear programming formulation of the maximum induced matching problem,
we introduce some notation.
Let $G$ be a graph.
For a vertex $u$ of $G$, let $\delta_G(u)$ be the set of edges of $G$ that are incident with $u$.
For an edge $uv$ of $G$, let 
$$\delta_G(uv)=\delta_G(u)\cup \delta_G(v),
\,\,\,\,\,\mbox{ and let }\,\,\,\,\,
C_G(uv)=\bigcup\limits_{w\in N_G[u]\cup N_G[v]}\delta_G(w).$$
Note that a set $M$ of edges in $G$ is an induced matching in $G$ 
\begin{itemize}
\item if and only if 
$f\not\in C_G(e)$ for every two distinct edges $e$ and $f$ in $M$
\item if and only if 
$\delta_G(e)$ contains at most one edge from $M$ for every edge $e$ of $G$. 
\end{itemize}
The second equivalence motivates the following 
(unweighted version of the) integer linear program from \cite{limeva}:
\begin{eqnarray}
\begin{array}{rrcll}
\max & \sum\limits_{e\in E(G)}x_e & & & \\
s.th. & \sum\limits_{f\in \delta_G(e)}x_f & \leq & 1 & \forall e\in E(G)\\
& x_e & \in & \{ 0,1\} & \forall e\in E(G)
\end{array}\label{e1}
\end{eqnarray}
Clearly, the value of (\ref{e1}) equals $\nu_s(G)$, 
and $\{ e\in E(G):x_e=1\}$ is an induced matching for every feasible solution $(x_e)_{e\in E(G)}$ of (\ref{e1}).
We consider the relaxation $(P)$ of (\ref{e1}) together with its dual linear program $(D)$:
\begin{eqnarray*}
(P)\left\{\begin{array}{rrcll}
\max & \sum\limits_{e\in E(G)}x_e & & & \\
s.th. & \sum\limits_{f\in \delta_G(e)}x_f & \leq & 1 & \forall e\in E(G)\\
& x_e & \geq & 0 & \forall e\in E(G)
\end{array}\right.
\hspace{2cm}
(D)\left\{\begin{array}{rrcll}
\min & \sum\limits_{e\in E(G)}y_e & & & \\
s.th. & \sum\limits_{f\in \delta_G(e)}y_f & \geq & 1 & \forall e\in E(G)\\
& y_e & \geq & 0 & \forall e\in E(G)
\end{array}\right.
\end{eqnarray*}
Let $\nu_s^*(G)$ and $\tau_s^*(G)$ denote the optimum values of the linear programs $(P)$ and $(D)$, respectively.
If $(x_e)_{e\in E(G)}$ is some feasible solution of $(P)$ or $(D)$, and $F\subseteq E(G)$, 
then let $x(F) = \sum\limits_{e \in F}{x_e}$,
in particular, $\sum\limits_{f\in \delta_G(e)}x_f=x\left(\delta_G(e)\right)$.

By linear programming duality, 
\begin{eqnarray}
\nu_s(G)\leq \nu_s^*(G)=\tau_s^*(G).\label{e3}
\end{eqnarray}
If $G$ is $\Delta$-regular, 
then setting $x_e=\frac{1}{2\Delta-1}$ for every edge $e$ of $G$ 
yields optimal solutions for $(P)$ and $(D)$
of value $\frac{m(G)}{2\Delta-1}$,
which implies that (\ref{e0}) follows immediately from (\ref{e3}).
The results of Lin et al. \cite{limeva} imply 
that the integrality gap of the weighted version of (\ref{e1}) is 
at most $\Delta$ and 
at least $\Delta-1$.

We conjecture that this can be improved considerably for unweighted graphs.

\begin{conjecture}\label{conjecture1}
If $G$ is a graph of maximum degree $\Delta$, then
\begin{eqnarray}\label{e4}
\frac{\nu_s^*(G)}{\nu_s(G)}\leq 
\begin{cases}
\frac{5\Delta^2}{8\Delta-4} & \mbox{, if $\Delta$ is even,}\\[3mm]
\frac{5\Delta^3-21\Delta^2+7\Delta+1}{8\Delta^2-36\Delta+20} & \mbox{, if $\Delta$ is odd}
\end{cases}
\end{eqnarray}
with equality in (\ref{e4}) if and only if $G$ arises by replacing the five vertices of 
the cycle $C_5$ of order five
with independent sets of cardinalities 
$\left\lfloor\frac{\Delta}{2}\right\rfloor$,
$\left\lfloor\frac{\Delta}{2}\right\rfloor$,
$\left\lfloor\frac{\Delta}{2}\right\rfloor$,
$\left\lceil\frac{\Delta}{2}\right\rceil$, and
$\left\lceil\frac{\Delta}{2}\right\rceil$ in this cyclic order.
\end{conjecture}
Note that the extremal graph in Conjecture \ref{conjecture1} 
also appears in Erd\H{o}s and Ne\v{s}et\v{r}il's famous open conjecture 
on the strong chromatic index \cite{fascgytu}. 
If $\Delta$ is even, 
then this blown-up $C_5$ is $\Delta$-regular, and $\nu_s^*(G)$ equals $\frac{m(G)}{2\Delta-1}=\frac{5\Delta^2}{8\Delta-4}$.
If $\Delta$ is odd, then setting 
\begin{itemize}
\item $x_e=\frac{\Delta-5}{2\Delta^2-9\Delta+5}$ for the edges $e$ between an independent set of order $\frac{\Delta-1}{2}$ and an independent set of order $\frac{\Delta+1}{2}$, 
and setting
\item $x_e=\frac{\Delta-3}{2\Delta^2-9\Delta+5}$ for all remaining edges $e$
\end{itemize}
yields optimal solutions for $(P)$ and $(D)$,
which explains the specific value in Conjecture \ref{conjecture1}.

Gotthilf and Lewenstein \cite{gole} 
obtain the approximation ratio $0.75\Delta+0.15$
by providing a polynomial time algorithm 
that computes an induced matching of size at least
\begin{eqnarray} 
\frac{m(G)}{1.5\Delta^2-0.5\Delta}\label{e5}
\end{eqnarray}
in a given not necessarily regular graph $G$ of maximum degree $\Delta$
(cf. also \cite{felera} choosing $f=3\Delta^2/2-\Delta/2$ and $g=0$ in Theorem 2(ii) and in the proof of Corollary 3).
For $\Delta$-regular graphs, it follows that the integrality gap of (\ref{e1})
is as most $\frac{1.5 \Delta^2 - 0.5\Delta}{2\Delta -1}=\frac{3}{4}\Delta+O(1)$.

We proceed to our results;
all proofs are postponed to the following sections.
Our first result is a best-possible upper bound on the fractional induced matching number.
Let $T^*$ be the tree that arises by subdividing each edge of 
the star $K_{1,\Delta}$ of order $\Delta+1$ once.
\begin{theorem} \label{t1}
If $G$ is a graph of maximum degree at most $\Delta$
such that no component of $G$ has order at most $2$, then
$\nu_s^*(G) \leq \frac{\Delta}{2\Delta +1} n(G)$
with equality if and only if each component of $G$ is isomorphic to $T^*$.
\end{theorem}
Combining Theorem \ref{t1} with the main result from \cite{jorasa} 
yields an approximation ratio of $18/7$ for subcubic graphs.
Our second result improves this.
\begin{theorem} \label{t2}
There is an efficient algorithm that,
for a given subcubic graph $G$,
produces
an induced matching $M$ in $G$ as well as 
a feasible solution $(y_e)_{e\in E(G)}$ of $(D)$
with 
$|M|\geq \frac{3}{7}y(E(G))$.
\end{theorem}
Our final result concerns general maximum degrees.
\begin{theorem} \label{t3}
There is an efficient algorithm that,
for a given graph $G$ of maximum degree at most $\Delta$ for some $\Delta \geq 3$,
produces
an induced matching $M$ in $G$ with 
$$|M|\geq \frac{\nu_s^*(G)}{(1-\epsilon)\Delta + \frac{1}{2}}\mbox{ where }\epsilon \approx 0.02005.$$
\end{theorem}
The last two theorems imply, in particular, 
that the problem to find a maximum induced matching 
in the considered graphs 
can be approximated in polynomial time within ratios of $\frac{7}{3}$
and $(1-\epsilon)\Delta + \frac{1}{2}$, respectively.
Theorem \ref{t3} allows an interesting corollary.
\begin{corollary}
If $G$ is a graph of maximum degree at most $\Delta$, then
\begin{align*}
 \nu_s(G) \geq \frac{\nu(G)}{2(1-\epsilon)\Delta +1}\mbox{ where }\epsilon \approx 0.02005,
\end{align*}
where $\nu(G)$ denotes the matching number of $G$.
\end{corollary}
\begin{proof}
Let $M$ be some maximum matching in $G$.
Setting $x_e=1/2$ for every edge $e$ in $M$,
and $x_e=0$ otherwise, yields a feasible solution of $(P)$.
This implies $\nu_s^*(G)\geq x(E(G))=|M|/2$.
Now, Theorem \ref{t3} implies the statement.
\end{proof}
We close the introduction with some notation.
Let $G$ be a graph.
We denote by $n(G)$, $m(G)$, and $\Delta(G)$
the order, size, and maximum degree of $G$,
respectively.
For a set $X$ of vertices of $G$, let 
$$N_G(X) = \bigcup_{u \in X}{N_G(u)}\setminus X,
\,\,\,\,\,\mbox{ and }\,\,\,\,\,
N_G[X] = X\cup \bigcup_{u \in X}{N_G(u)}.$$
For two disjoint sets $X$ and $Y$ of vertices of $G$,
let $G[X]$ be the subgraph of $G$ induced by $X$,
$$E_G(X,Y) = \{uv \in E(G) : u \in X \mbox{ and } v \in Y\},$$
$E_G(X) = E(G[X])$,
$m_G(X,Y) = |E_G(X,Y)|$, and
$m_G(X) = |E_G(X)|$, respectively.
Finally, for a set $F$ of edges of $G$,
let $V(F)$ be the set of vertices of $G$
that are incident to some edge in $F$.

\section{Proof of Theorem \ref{t1}}
Let $G$ be a graph of maximum degree $\Delta$
such that no component of $G$ has order at most $2$,
in particular, $\Delta\geq 2$.
Note that $\nu_s^*(T^*)=\Delta=\frac{\Delta}{2\Delta +1}n(T^*)$.
Therefore, by (\ref{e3}),
it suffices to show the existence of a feasible solution $(y_e)_{e\in E(G)}$ of $(D)$ with
\begin{enumerate}[(i)]
\item $y(\delta_G(u)) \geq \frac{1}{2}$ for every vertex $u$ of degree less than $\Delta$, and
\item $y(E(G)) \leq \frac{\Delta}{2\Delta +1}n(G)$,
\end{enumerate}
such that (ii) holds with equality if and only if every component of $G$ is isomorphic to $T^*$.
We call such a feasible solution {\it good},
and we show the existence of a good solution by induction on the order $n(G)$.
Since the considered quantities are all additive with respect to the components,
we may assume that $G$ is connected.
If $G$ is $\Delta$-regular, then setting $y_e=\frac{1}{2\Delta -1}$ for every edge $e$ of $G$ yields a good solution.
Hence, we may assume that the minimum degree of $G$ is less than $\Delta$. 

Let $u$ be a vertex of minimum degree $\delta$.

\bigskip

\noindent \textbf{Case 1.} {\it $\delta=1$, and no component of $G-N_G[u]$ has order at most $2$.}

\bigskip

\noindent Let $v$ be the unique neighbor of $u$ in $G$,
let $w$ be some neighbor of $v$ distinct from $u$, and let $G^\prime = G - \{u,v,w\}$.
Let $I_1$ be the set of isolated vertices in $G^\prime$, 
and let $I_2$ be the set of vertices of the components of order $2$ in $G^\prime$. 
By induction, there is a good solution $(y^\prime_e)_{e\in E(G^\prime - (I_1 \cup I_2))}$ for $G^\prime - (I_1 \cup I_2)$.
Since the graph $G-\{u,v\}$ has no component of order at most $2$, 
each component of $G[I_1 \cup I_2]$ sends at least one edge to $w$. Hence, we obtain 
$|I_1| + \frac{|I_2|}{2} \leq  d_G(w) - 1 \leq  \Delta - 1$, which implies
\begin{align} \label{comp_leq2:case1}
|I_1| + |I_2|\leq 2|I_1| + |I_2| \leq  2(\Delta - 1).
\end{align}
Let $M$ be a set of edges in $E_G(\{w\}, I_1 \cup I_2)$ 
such that each component of $G[I_1 \cup I_2]$ is incident with exactly one edge in $M$.
Let
\begin{eqnarray*}
 y_e = 
 \begin{cases}
    y^\prime_e 	 & \mbox{, if } e \in E(G^\prime - (I_1 \cup I_2)) \mbox{,}\\
    \frac{1}{2}   	& \mbox{, if }  e \in M \cup E_G(I_2) \cup \{uv, vw\} \mbox{, and}\\
    0                 & \mbox{, otherwise.} 
 \end{cases} 
\end{eqnarray*}
It is easy to see that $(y_e)_{e\in E(G)}$ is a feasible solution of $(D)$ that satisfies (i).
By induction and (\ref{comp_leq2:case1}), we obtain that
\begin{align*}
y(E(G)) &= y'(E(G^\prime - (I_1 \cup I_2))) + \frac{1}{2}(|I_1|+|I_2| + 2) \\
	     &\leq \frac{\Delta}{2\Delta + 1}\Big(n(G) - (|I_1|+|I_2| +3)\Big) + \frac{1}{2}(|I_1|+|I_2| + 2) \\
	     &= \frac{\Delta}{2\Delta + 1}n(G) + \frac{|I_1| + |I_2| - 2(\Delta - 1)}{2 (2\Delta+1)} \\
	     &\leq \frac{\Delta}{2\Delta + 1}n(G).
\end{align*} 
Now, we assume that $y(E(G))= \frac{\Delta}{2\Delta + 1}n(G)$, 
and that $G$ is not isomorphic to $T^*$.
Equality in the above inequality chain implies $|I_1| + |I_2| = 2(\Delta -1)$, 
which, by (\ref{comp_leq2:case1}),
implies that $I_1 = \emptyset$, $|I_2| = 2 (\Delta -1)$, and $d_G(w) = \Delta$.
It follows that each component of order $2$ in $G^\prime$ sends exactly one edge to $w$. 
Moreover, 
$$y'(E(G^\prime - (I_1 \cup I_2)))=\frac{\Delta}{2\Delta + 1}\Big(n(G) - (|I_1|+|I_2| +3)\Big),$$
which implies that 
every component of $G^\prime - (I_1 \cup I_2)$ is isomorphic to $T^*$.
Since $G$ is not isomorphic to $T^*$, 
the vertex $v$ has a third neighbor $y$ distinct from $u$ and $w$,
which either belongs 
to $I_2$ 
or 
to $V(G^\prime)\setminus (I_1 \cup I_2)$.

First, we assume that $y \in I_2$. 
Let $M$ be a set of edges in $E_G(\{v,w\}, I_1 \cup I_2)$
such that each component of $G[I_1 \cup I_2]$ is incident with exactly one edge in $M$,
and $vy\in M$.
Let
\begin{eqnarray*}
 y_e = 
 \begin{cases}
    y^\prime_e 	& \mbox{, if }  e \in E(G^\prime - (I_1 \cup I_2)) \mbox{,}\\
    \frac{1}{2}   	& \mbox{, if }  e \in M \cup E_G(I_2) \cup \{uv\} \mbox{, and}\\
    0               & \mbox{, otherwise.} 
 \end{cases} 
\end{eqnarray*}
It is easy to see that $(y_e)_{e\in E(G)}$ is a feasible solution of $(D)$ that satisfies (i). 
As above, we obtain 
\begin{align*}
y(E(G)) &\leq \frac{\Delta}{2\Delta + 1}n(G) + \frac{|I_1| + |I_2| - 2(\Delta - 1)}{2 (2\Delta+1)} - \frac{1}{2}
	     <\frac{\Delta}{2\Delta + 1}n(G),	     
\end{align*} 
which completes the proof in this case.

Hence, we may assume that $v$ has no neighbor in $I_2$,
which implies $y\in V(G^\prime)\setminus (I_1 \cup I_2)$.
The component $H$ of $G^\prime - (I_1 \cup I_2)$ containing $y$ is isomorphic to $T^*$.
Let $f$ be the unique edge of $H$
that is incident to the vertex of degree $\Delta$ in $H$,
and has minimum distance to $y$.
Let
\begin{eqnarray*}
 y_e = 
 \begin{cases}
    y^\prime_e 	& \mbox{, if }  e \in E(G^\prime - (I_1 \cup I_2))\setminus E(H) \mbox{,}\\
   \frac{1}{2}   	& \mbox{, if }  e \in E_G(I_2) \cup \big(\delta_G(w) \setminus\{vw\}\big) 
   \cup \{uv,vy\} \cup \big(E(H)\setminus \{ f\}\big) \mbox{, and} \\
   0                    & \mbox{, otherwise.} 
 \end{cases} 
\end{eqnarray*}
It is easy to see that $(y_e)_{e\in E(G)}$ is a feasible solution of $(D)$ that satisfies (i). 
Furthermore, 
\begin{align*}
 y(E(G)) \leq \frac{\Delta}{2\Delta + 1} \left(n(G) - (4\Delta + 2) \right) + 2\Delta  - \frac{1}{2} <\frac{\Delta}{2\Delta +1} n(G), 
\end{align*}
which completes the proof in the first case.

\bigskip

\noindent \textbf{Case 2.} {\it Case 1 does not apply.}

\bigskip

\noindent Let $G^\prime = G-N_G[u]$,
let $I_1$ be the set of isolated vertices in $G^\prime$, 
and let $I_2$ be the set of vertices of the components of order $2$ in $G^\prime$. 
Note that $G^\prime$ differs slightly from Case 1.
Let $(y^\prime_e)_{e\in E(G^\prime - (I_1 \cup I_2))}$ be a good solution for $G^\prime - (I_1 \cup I_2)$.
If $\delta =1$, then each component of $G[I_1 \cup I_2]$ sends at least one edge into $N_G(u)$.
If $\delta \geq 2$, then each vertex in $I_1$ sends at least $\delta$ many edges into $N_G(u)$
while each vertex in $I_2$ sends at least $\delta-1$ edges into $N_G(u)$.
Hence, we obtain that
\begin{align} \label{comp_leq2:case2}
 \delta |I_1| + \max\left\{ (\delta-1)|I_2|, \frac{|I_2|}{2} \right\} \leq (\Delta-1) \delta.
\end{align}
If $\delta = 1$, then this implies that $|I_1| + |I_2| \leq 2(\Delta -1)<2\Delta-\delta$.
If $\delta \geq 2$, then this implies that 
$$|I_1| + |I_2| \leq 
\frac{\delta}{\delta-1}|I_1|+\frac{\delta-1}{\delta-1}|I_2|\leq 
(\Delta-1) \frac{\delta}{\delta-1} \leq 2\Delta - \delta,$$
and, since $\Delta>\delta$, 
equality in this last inequality chain only holds if 
$\delta=2$,
$I_1=\emptyset$, and 
$|I_2|=2(\Delta-1)$.

Let $M$ be a set of edges in $E_G(N_G(u), I_1 \cup I_2)$ such that
each component of $G[I_1 \cup I_2]$ is incident with exactly one edge in $M$.
Let
\begin{eqnarray*}
y_e =
 \begin{cases}
y^\prime_e 	& \mbox{, if }  e \in E(G^\prime - (I_1 \cup I_2)) \mbox{,}\\
\frac{1}{2}   	& \mbox{, if }  e \in M \cup E_G(I_2) \cup \delta_G(u) \mbox{, and}\\
0               & \mbox{, otherwise.} 
\end{cases} 
\end{eqnarray*}
For $\delta \geq 2$, it is easy to see that $(y_e)_{e\in E(G)}$ is a feasible solution of $(D)$ 
that satisfies (i). 
If $\delta = 1$, then, in view of Case 1, 
the set $I_1 \cup I_2$ is non-empty, which implies that the unique neighbor, say $v$, 
of $u$ is incident with an edge $vw$ such that $w \in I_1 \cup I_2$ and $y_{vw} = \frac{1}{2}$.
Hence, also in this case, 
$(y_e)_{e\in E(G)}$ is a feasible solution of $(D)$ that satisfies (i).
By induction and (\ref{comp_leq2:case2}), we obtain that
\begin{align*}
y(E(G)) 
	     &\leq \frac{\Delta}{2\Delta + 1}\Big(n(G) - (|I_1|+|I_2| + \delta + 1)\Big) + \frac{1}{2}(|I_1|+|I_2| + \delta) \\
	     &= \frac{\Delta}{2\Delta + 1}n(G) + \frac{|I_1| + |I_2| - (2\Delta - \delta)}{2 (2\Delta+1)} \\
	     &\leq \frac{\Delta}{2\Delta + 1}n(G).
\end{align*}
Now, we assume that $y(E(G)) = \frac{\Delta}{2\Delta + 1}n(G)$, 
and that $G$ is not isomorphic to $T^*$.
Equality in the above inequality chain implies $|I_1| + |I_2| = 2\Delta - \delta$.
As observed above, 
this implies $\delta = 2$, $I_1 = \emptyset$, and $|I_2| = 2(\Delta -1)$. 
It follows that every vertex in $I_2$ has degree exactly $2$,
every vertex in $N_G(u)$ has degree $\Delta$,
and $G'-(I_1\cup I_2)$ is empty.
Let $f$ be an edge incident with $u$.
Let
\begin{eqnarray*}
y_e = 
\begin{cases}
\frac{1}{2}& \mbox{, if } e \in E_G(N_G(u),I_2)\cup \big(\delta_G(u)\setminus \{ f\}\big) \mbox{, and}\\ 
0 &\mbox{, otherwise.}
\end{cases}
\end{eqnarray*}
Since $\Delta \geq 3$, 
it follows that $(y_e)_{e\in E(G)}$ is a feasible solution of $(D)$ that satisfies (i).  
Furthermore, 
$$y(E(G)) =  \Delta - \frac{1}{2} < 
\frac{\Delta}{2\Delta +1} (2\Delta+1) = \frac{\Delta}{2\Delta +1} n(G),$$ 
which completes the proof.

\section{Proof of Theorem \ref{t2}}

Let $G$ be a graph of maximum degree at most $3$.
In view of the $\frac{9}{5}$-approximation algorithm for cubic graphs given in \cite{jorasa},
we may assume that $G$ is not cubic.
Clearly, we may assume that $G$ has no isolated vertices.
We will describe an efficient recursive algorithm 
that constructs an induced matching $M$ in $G$
together with a feasible solution $(y_e)_{e\in E(G)}$ 
of the linear program $(D)$ such that
\begin{enumerate}[(i)]
\item $y(\delta_G(u)) \geq \frac{1}{3}$ for every vertex $u$ of degree at most $2$ in $G$, and
\item $y(E(G)) \leq \frac{7}{3} |M|$.
\end{enumerate}
We call the pair $\big(M,(y_e)_{e\in E(G)}\big)$ a {\it good solution pair for} $G$.

The algorithm performs the following steps:
\begin{itemize}
\item[(1)] {\it Select an edge $v_0v_1$ of $G$ incident with a vertex $v_0$ 
of minimum degree.} 
\end{itemize}
In view of the above assumptions, 
the minimum degree $d_G(v_0)$ is either $1$ or $2$.
Let $I$ be the set of isolated vertices of 
$G-\big(N_G[v_0]\cup N_G[v_1]\big)$, 
let $H=G\big[N_G[v_0]\cup N_G[v_1]\cup I\big]$,
and let $G'=G-V(H)$.
\begin{itemize}
\item[(2)] {\it Apply the algorithm recursively to $G'$
to obtain a good solution pair 
${\cal P}'=\big(M',(y_e)_{e\in E(G')}\big)$ for $G'$.} 
\item[(3)] {\it Set $M$ equal to $M'\cup \{ v_0v_1\}$.}
\end{itemize}
Note that $M$ is an induced matching in $G$ by construction.
\begin{enumerate}
\item[(4)] {\it Specify values $y_e$ for all edges $e$ 
in $E(G)\setminus E(G')$ such that:}
\begin{enumerate}[(a)]
\item {\it $y\left(\delta_H(e)\right)\geq 1$ for every edge $e$ of $H$.}
\item {\it $y(\delta_H(u)) \geq \frac{2}{3}$ for every vertex $u$ 
in $\big(N_G[v_0]\cup N_G[v_1]\big)\setminus \{ v_0,v_1\}$
of degree at most $2$ in $H$.}
\item {\it $y(\delta_H(u))\geq \frac{1}{3}$ for every vertex $u$ 
in $\{ v_0,v_1\}\cup I$ of degree at most $2$ in $H$.}
\item {\it $y_e=0$ for all edges of $G$ between $V(H)$ and $V(G')$.}
\item {\it $y\left(E(H)\right)\leq \frac{7}{3}$.}
\end{enumerate}
\end{enumerate}
Condition (i) for the good solution pair ${\cal P}'$ for $G'$
and condition (b) together imply 
$y\left(\delta_G(e)\right)\geq 1$ for every edge $e$
of $G$ between $V(H)$ and $V(G')$.
This together with condition (a) implies that 
$(y_e)_{e\in E(G)}$ is a feasible solution of $(D)$.
Condition (ii) for ${\cal P}'$ and
conditions (b) to (e) together 
imply conditions (i) and (ii) for the pair 
${\cal P}=\big(M,(y_e)_{e\in E(G)}\big)$.
Altogether, it follows that ${\cal P}$
is a good solution pair for $G$.

It remains to show that step (4) is possible, 
more precisely, that the values $y_e$ can be specified 
in such a way that conditions (a) to (e) hold.
We show this by considering all $56$ 
possibilities for the structure of $H$
shown in Figures \ref{figconfig1} and \ref{figconfig2}.
Figure \ref{figconfig1} shows the $24$ 
possibilities with $d_H(v_0)+d_H(v_1)<2+3$,
and Figure \ref{figconfig2}
shows the remaining $22$
possibilities with $d_H(v_0)+d_H(v_1)=2+3$.
Since $v_0$ is a vertex of minimum degree,  
each vertex in $I$ has at least $d_G(v_0)$ 
neighbors in $\big(N_G[v_0]\cup N_G[v_1]\big)\setminus \{ v_0,v_1\}$.
Within the figures, 
we also show suitable values for the $y_e$
satisfying conditions (a) to (e),
multiplied by $3$ for the sake of readability.

It is a tedious yet routine matter to verify 
that the figures show all possibilities for the structure of $H$,
and that the conditions (a) to (e) indeed hold.
Steps (1), (3), and (4) can clearly be performed in polynomial time,
which yields an overall polynomial running time,
and completes the proof of Theorem \ref{t2}.

\begin{figure}[H]
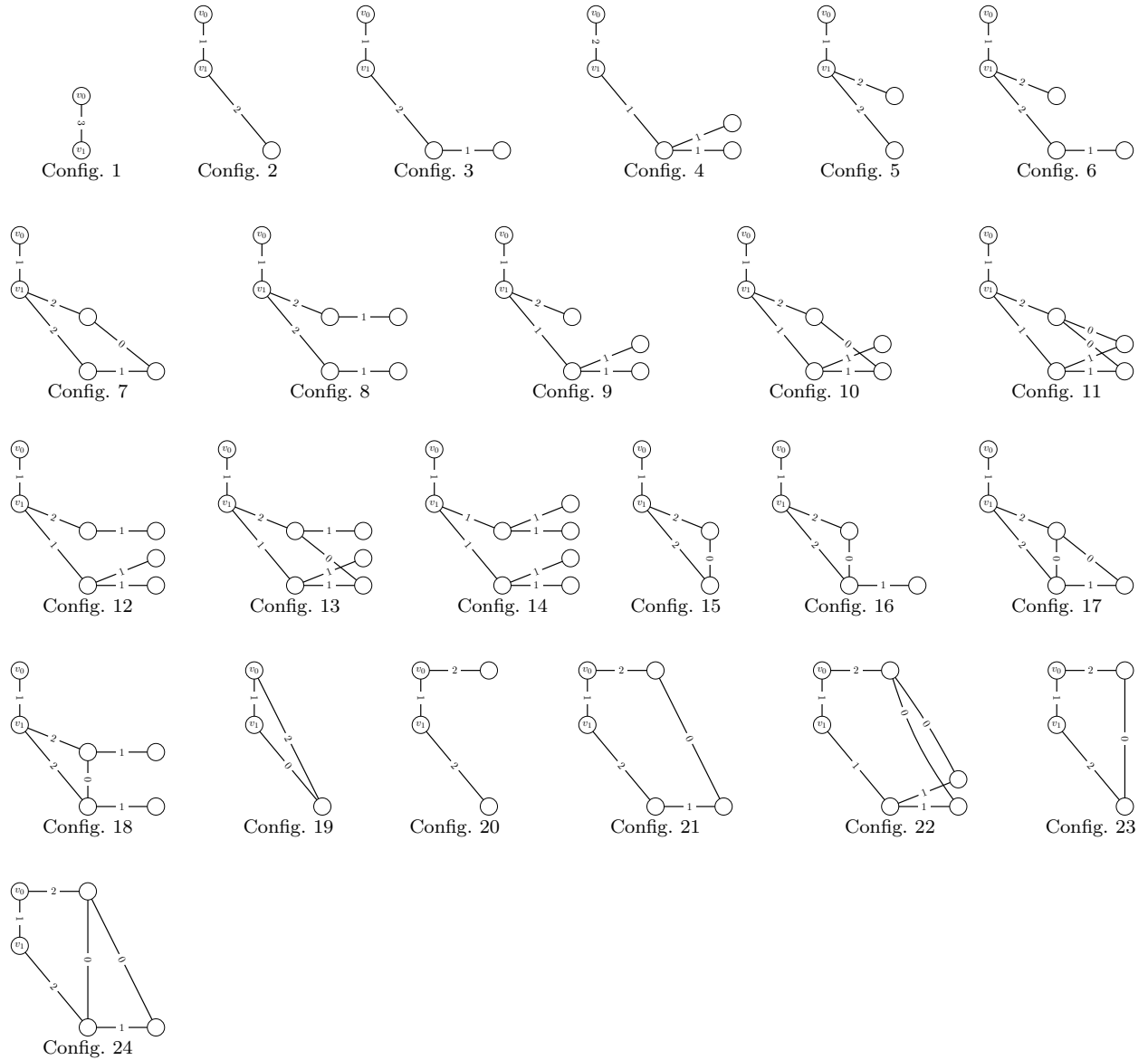

 \hspace{0.5cm}\vspace{0.1cm} 
 $\underset{\text{Config.\ 1}}{ 
 \scalebox{0.4}{

}}$
 \caption{The $24$ possibilities for the structure of $H$ 
with $d_H(v_0)+d_H(v_1)<2+3$
together with the values of $3y_e$ for all edges $e$ of $H$.}\label{figconfig1}
\end{figure}

\pagebreak 

\begin{figure}[H]
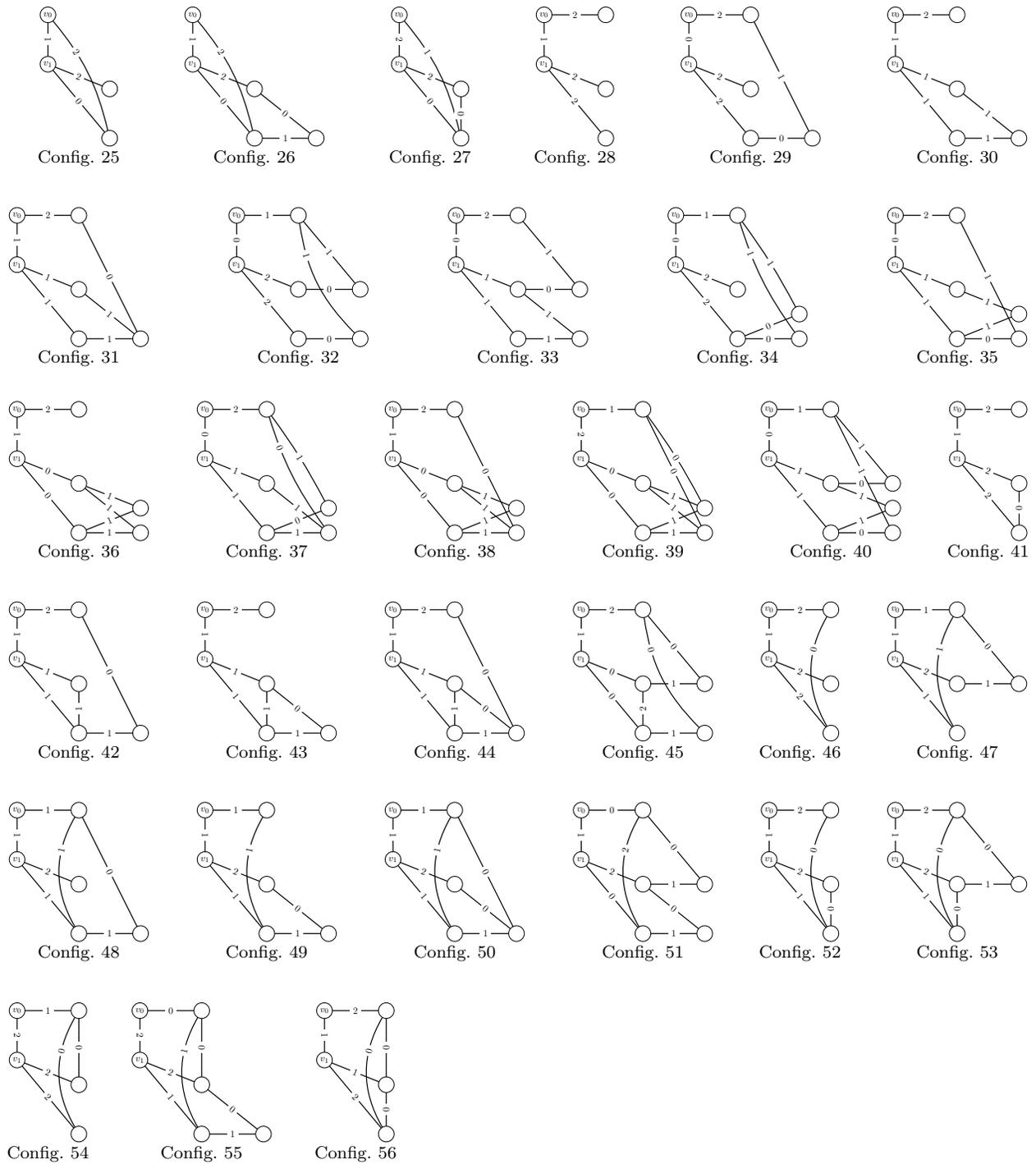

 \hspace{0.5cm}\vspace{0.1cm} 
 $\underset{\text{Config.\ 25}}{ 
 \scalebox{0.4}{
 %
}}$
\caption{The $22$ possibilities for the structure of $H$ 
with $d_H(v_0)+d_H(v_1)=2+3$
together with the values of $3y_e$ for all edges $e$ of $H$.}\label{figconfig2}
\end{figure}

The conditions (i) in the proofs of Theorem \ref{t1} and Theorem \ref{t2} are very similar;
they both allow that good solutions/good solution pairs for the considered subgraphs  
do not have to be changed when constructing the solution for the entire graph.
Unfortunately, for larger values of $\Delta$,
this approach seems not to lead to improved approximation ratios.
Suppose that we impose a condition like $y(\delta_G(u)) \geq c$ for some positive constant $c$
and every vertex $u$ of degree less than $\Delta$.
If $G$ arises from $K_{\lceil \frac{\Delta}{2} \rceil+1}$
by attaching $\lfloor \frac{\Delta}{2} \rfloor$ many leaves to each vertex,
then $\nu_s(G) = 1$ while any solution satisfying the above condition 
has total weight at least $c\left(\lceil \frac{\Delta}{2} \rceil+1\right) \lfloor \frac{\Delta}{2} \rfloor$.
This is also not surprising as combining the best-possible lower bound on the 
induced matching number for graphs of bounded maximum degree \cite{jo}
with the best-possible Theorem \ref{t1} 
does not lead to an improved approximation ratio;
both results are tight for different graphs.

\section{Proof of Theorem \ref{t3}}
For simplification, we are not trying to optimize the non-leading terms of the approximation ratio.
For the rest of this section,
let $(\epsilon, c)$ be an optimal solution of the following quadratic program:
\begin{eqnarray*}
(Q)\left\{ 
\begin{array}{lrcl}
\max & \epsilon &  &   \\[2mm]
s.th.  & 1.5\left( 1+ \frac{\epsilon (2c-1+\epsilon)}{1-c - \epsilon}\right) &\leq & 2c(1-\epsilon) \\
 & \epsilon &\leq & (1-c)^2 \\ 
 & \epsilon + c & < & 1 \\
 & \epsilon, c & > & 0
\end{array}\right. \label{e7}
\end{eqnarray*} 
Standard software yields 
$$\epsilon \approx 0.02005\,\,\,\,\,\,\,\,\,\,\,\,\mbox{ and }\,\,\,\,\,\,\,\,\,\,\,\,c \approx 0.85838.$$
Note that verifying that these values yield a feasible solution for $(Q)$ is a simple matter of calculation, and, in fact, all our arguments only use the feasible of this solution.
Throughout this section, the parameter $f$ is chosen as follows:
$$f = (1-\epsilon)\Delta + \frac{1}{2}\approx 0.97995\Delta+0.5.$$
A key ingredient for the proof of Theorem \ref{t3} is the following lemma.
\begin{lemma} \label{l1}
If $G$ is a graph of maximum degree at most $\Delta$ for some $\Delta\geq 3$, 
and $(x_e)_{e \in E(G)}$ is a feasible solution for $(P)$ that satisfies 
$x(C_G(e)) \geq f$ for every edge $e$ of $G$,
then
\begin{align*}
x(E(G)) \leq  \frac{(1-\epsilon)m(G)}{1.5\Delta}.
\end{align*}
\end{lemma}
We postpone the proof of Lemma \ref{l1} to the end of this section.
The condition in Lemma \ref{l1} 
will be ensured by the following {\sc Local Ratio Preprocessing}, 
which is similar to the technique used in \cite{limeva}.
Note that {\sc Local Ratio Preprocessing} needs to solve the linear program $(P)$ only once,
while \cite{limeva} requires to solve a linear program in each iteration.

\bigskip

\begin{algorithm}[H]
\LinesNumbered\SetAlgoLined%\linesnumbered\SetLine 
\KwIn{A graph $G$.}
\KwOut{An induced matching $M$ and a subgraph of $G$}
\Begin{
\LinesNumbered\SetAlgoLined
$M\leftarrow \emptyset$\;
Let $(x_e)_{e \in E(G)}$ be an optimal solution of $(P)$\label{line3}\;
\While{$G$ has an edge $e$ satisfying $x(C_G(e)) \leq f$}
{
$M\leftarrow M\cup \{ e\}$; $G\leftarrow G-C_G(e)$\;
}
\Return $(M,G)$\;
}
\caption{{\sc Local Ratio Preprocessing}} \label{alg1}
\end{algorithm}

\bigskip

\begin{proof}[Proof of Theorem \ref{t3}]
Let $G$ be as in the statement of the theorem.
Applying {\sc Local Ratio Preprocessing} to $G$ 
produces in polynomial time an output $(M,G^\prime)$, 
where $G^\prime = G - \bigcup\limits_{e\in M}{C_G(e)}$.
Furthermore, 
the restriction to $E(G')$
of the optimal solution $(x_e)_{e \in E(G)}$ of $(P)$ 
chosen in line \ref{line3}
is a feasible solution for the linear program $(P)$ on the graph $G'$
that satisfies $x(C_{G'}(e))\geq f$ for every edge $e$ of $G'$.
Let $M = \{e_1,\ldots,e_k\}$,
where the edges were added to $M$ in the order $e_1,\ldots,e_k$.
The choice of the edges within {\sc Local Ratio Preprocessing} implies
$x\left(C_G(e_j) \setminus \bigcup\limits_{i=1}^{j-1}{C_G(e_i)}\right)\leq f$
for every $j$ in $[k]$,
and, hence, 
\begin{align*}
f |M| 
\geq \sum_{j=1}^k{x\left(C_G(e_j) \setminus \bigcup_{i=1}^{j-1}{C_G(e_i)}\right)} 
=x\left(\bigcup_{j=1}^k\left(C_G(e_j) \setminus \bigcup_{i=1}^{j-1}{C_G(e_i)}\right)\right) 
=x\left(\bigcup_{j=1}^k{C_G(e_j)}\right).
\end{align*}
Let $M^\prime$ be an induced matching in $G^\prime$ of size at least $\frac{m(G^\prime)}{1.5\Delta^2}$, cf. (\ref{e5}).
As noted in the introduction, 
we can find such an induced matching in polynomial time \cite{gole, felera}.
By construction, 
$M \cup M^\prime$ is an induced matching in $G$.
The choice of $f$ and Lemma \ref{l1} imply
$$|M'|\geq \frac{m(G^\prime)}{1.5\Delta^2}\geq \frac{x(E(G'))}{(1-\epsilon)\Delta} \geq \frac{x(E(G'))}{f},$$ 
and, hence,
\begin{align*}
|M \cup M^\prime| 
&\geq \frac{1}{f} \left( x\left(\bigcup_{e \in M}{C_G(e)}\right) + x(E(G^\prime)) \right)
=\frac{1}{f} x(E(G))
= \frac{\nu_{s}^*(G)}{f},
\end{align*}
which completes the proof.
\end{proof}
We proceed to the proof of Lemma \ref{l1}.
\begin{proof}[Proof of Lemma \ref{l1}]
For notational convenience, we introduce one further parameter: $$g = \frac{\epsilon}{1-c}\approx 0.14158.$$
Let 
$$I = \left\{u \in V(G) : d_G(u) < c \Delta + \frac{1}{2} \right\}.$$
\begin{claim} \label{c1}
If $u$ is in $I$ and $v$ is a neighbor of $u$, then 
$$d_G(v) \geq (1-g) \Delta + 1\mbox{ and }x(\delta_G(v)) \leq g.$$
\end{claim}
\begin{proof}[Proof of Claim \ref{c1}]
Let $v^\prime$ be a neighbor of $u$ maximizing $x\left(\delta_G(v^\prime)\right)$.
For every neighbor $w$ of $v'$,
the constraints in $(P)$ imply
$$1\geq x\left(\delta_G(v'w) \right)
=x\left(\delta_G(v')\right)+x\left(\delta_G(w)\setminus \{ v'w\}\right),$$
which implies
$x\left(\delta_G(w)\setminus \{ v'w\}\right)\leq 1-x\left(\delta_G(v')\right)$
and $x\left(\delta_G(v')\right)\leq 1$.

Now, 
\begin{align*}
f & \leq x\left(C_G(uv^\prime)\right)\\
& \leq
\sum_{w\in N_G(u)}x\left(\delta_G(w)\right)
+\sum_{w\in N_G(v')\setminus \{ u\}}x\left(\delta_G(w)\setminus \{ v'w\}\right)\\
&\leq 
d_G(u) x\left(\delta_G(v^\prime)\right)
+(\Delta-1)\left(1-x\left(\delta_G(v^\prime)\right)\right)\\
&\leq c \Delta x\left(\delta_G(v^\prime)\right) + \frac{1}{2}
+\Delta\left(1-x\left(\delta_G(v^\prime)\right)\right)\\
&= \Delta\left(1-(1-c)x\left(\delta_G(v^\prime)\right)\right) + \frac{1}{2},
 \end{align*}
which, by the definitions of $f$ and $g$, implies 
$$x(\delta_G(v)) \leq x(\delta_G(v^\prime)) \leq g.$$
Since $g < 1$, we obtain that
\begin{align*}
f & \leq x(C_G(uv))\\
& \leq
\sum_{w\in N_G(u)}x\left(\delta_G(w)\right)
+\sum_{w\in N_G(v)\setminus \{ u\}}x\left(\delta_G(w)\setminus \{ vw\}\right)\\
& \leq d_G(u)g+(d_G(v)-1)\\
& \leq \left( c \Delta + \frac{1}{2}\right)g+d_G(v) - 1\\
& \leq d_G(v) + g c \Delta - \frac{1}{2}.
\end{align*}
By the definitions of $f$ and $g$, this implies that 
$$d_G(v) \geq 
f-g c \Delta+\frac{1}{2}
=\Delta(1-\epsilon - g c) +1 = \Delta(1-g) + 1,$$
which completes the proof of the claim.
\end{proof}
The condition $\epsilon\leq (1-c)^2$ within (Q) implies $1-g\geq c$.
Therefore, Claim \ref{c1} implies that the set $I$ is independent, and that
\begin{align}
 x\left(E_G(I,N_G(I))\right) &\leq \sum_{v \in N_G(I)}{x(\delta_G(v))}\nonumber \\
				&\leq g |N_G(I)|\nonumber \\
				&\leq \frac{g}{\Delta(1-g)} \sum_{v \in N_G(I)}{d_G(v)} \nonumber\\
				&\leq \frac{g}{\Delta(1-g)}\Big(m(G) + m_G(N_G(I))\Big).\label{eneu1}
\end{align}
By Claim \ref{c1}, each edge $uv$ in $E_G(N_G(I))$ satisfies 
$$x\left(\delta_G(uv)\right) 
=x\left(\delta_G(u)\right) + x\left(\delta_G(v)\right) -x_{uv}\leq 2g.$$
Note that $e'\in \delta_G(e)$ if and only if $e\in \delta_G(e')$
for every two edges $e$ and $e'$ of $G$, and double-counting implies
$$\sum_{e \in E(G)}{x(\delta_G(e))}=\sum_{e \in E(G)}{x_e |\delta_G(e)|}.$$
For every edge $e$ of $G$, 
we obtain that
\begin{align*}
f \leq x\left(C_G(e)\right) \leq \sum_{e^\prime \in \delta_G(e)}{x\left(\delta_G(e^\prime)\right)}  \leq |\delta_G(e)|.
\end{align*}
If $uv$ is an edge in $E(G)\setminus E_G(I,N_G(I))$,
then, by the definition of $I$,
$$|\delta_G(uv)|=|\delta_G(u)|+|\delta_G(v)|-1
\geq c\Delta+\frac{1}{2}+c\Delta+\frac{1}{2}-1
=2c\Delta.$$
Combining these four observations, we obtain  
\begin{align}
m(G)-(1-2g)m_G(N_G(I)) 
&= (m(G)-m_G(N_G(I)))+2gm_G(N_G(I))\nonumber \\
&=\sum_{e \in E(G)\setminus E_G(N_G(I))}1+\sum_{e \in E_G(N_G(I))}2g\nonumber  \\
&\geq \sum_{e \in E(G)\setminus E_G(N_G(I))}{x(\delta_G(e))}+\sum_{e \in E_G(N_G(I))}{x(\delta_G(e))}\nonumber  \\
&=\sum_{e \in E(G)}{x(\delta_G(e))}\nonumber  \\
&=\sum_{e \in E(G)}{x_e |\delta_G(e)|}\nonumber  \\
&=\sum_{e \in E_G(I,N_G(I))}{x_e |\delta_G(e)|}
+\sum_{e \in E(G)\setminus E_G(I,N_G(I))}{x_e |\delta_G(e)|}\nonumber  \\
&\geq f x(E_G(I,N_G(I))) + 2c \Delta x(E(G)\setminus E_G(I,N_G(I))).\label{eneu2}
\end{align}
Since $2c\approx 1.71676$ and $\Delta\geq 3$, we have $2c\Delta \geq f$.
Furthermore,
$$2g+\frac{(2c-1+\epsilon)g}{1-g}\approx 0.40467<1,$$ and, by $(Q)$ and the definition of $g$,
$$\left(1 + \frac{(2c-1+\epsilon)g}{1-g}\right)
=\left(1 + \frac{\epsilon(2c-1+\epsilon)}{1-c-\epsilon}\right)
\leq \frac{2c(1-\epsilon)}{1.5}.$$
Using these inequalities together with (\ref{eneu1}) and (\ref{eneu2}) yields 
\begin{align*}
2c\Delta x(E(G)) 
&= \Big(f x(E_G(I,N_G(I)))+2c\Delta x(E(G)\setminus E_G(I,N_G(I)))\Big)
+(2c\Delta-f) x(E_G(I,N_G(I)))\\
&\leq \Big(m(G) - (1-2g)m(N_G(I))\Big) + \frac{(2c\Delta - f)g}{\Delta(1-g)}(m(G) + m(N_G(I))) \\ 
&= m(G)\left(1 + \frac{(2c\Delta-f)g}{\Delta(1-g)}\right) 
+m(N_G(I))\left(2g+\frac{(2c\Delta-f)g}{\Delta(1-g)}-1\right) \\						
&\leq m(G)\left(1 + \frac{(2c-1+\epsilon)g}{(1-g)}\right) 
+m(N_G(I)) \left(2g+\frac{(2c-1+\epsilon)g}{(1-g)}-1\right)\\
&\leq m(G)\left(1 + \frac{(2c-1+\epsilon)g}{(1-g)}\right)\\
&\leq \frac{2c(1-\epsilon)}{1.5}m(G), 				    
\end{align*}
which completes the proof of Lemma \ref{l1}.
\end{proof}
Using this proof technique, what is the largest $\epsilon$
we can hope for? Let $\epsilon = \frac{1}{7} + \frac{1}{\Delta}$,
and let $\Delta$ be sufficiently large such that $0.75 \Delta$ is an integer.
Let $G$ be a bipartite graph with partite set $A$ and $B$
such that each vertex in $A$ has degree $0.75 \Delta + 2$,
each vertex in $B$ has degree $\Delta$, and $G$
has girth at least $6$.
It is well-known that such graphs exist, cf. e.g. \cite{fu}.
Setting $x_e = \frac{1}{1.75 \Delta + 1}$ for every edge $e$ of $G$
yields a feasible solution of both program $(P)$ and $(D)$,
that is, $(x_e)_{e \in E(G)}$ is optimal.
The girth condition and 
$(1.75 \Delta + 1) (1 - \epsilon) < 1.5 \Delta$
imply that
$$x(C_G(e)) \geq \frac{2 \Delta(0.75 \Delta +2)}{1.75\Delta + 1}  - 1 
	   > \frac{2(0.75 \Delta +2)(1-\epsilon)}{1.5} - 1
	   = (1-\epsilon)\Delta + \frac{5 - 8 \epsilon}{3}
	   > (1-\epsilon)\Delta + \frac{1}{2}$$
for sufficiently large $\Delta$.
Furthermore, $\nu_s^*(G) \geq \frac{m(G)}{1.75 \Delta + 1} > \frac{(1-\epsilon)m(G)}{1.5\Delta}$,
that is, Lemma \ref{l1} fails whenever $\epsilon \geq \frac{1}{7} + \frac{1}{\Delta}$.


\begin{thebibliography}{}
\bibitem{dadelo} K.K. Dabrowski, M. Demange, V.V. Lozin, New results on maximum induced matchings in bipartite graphs and beyond, Theoretical Computer Science 478 (2013) 33-40.
\bibitem{dumazi} W. Duckworth, D.F. Manlove, M. Zito, On the approximability of the maximum induced matching problem, Journal of Discrete Algorithms 3 (2005) 79-91.
\bibitem{fascgytu} R.J. Faudree, R.H. Schelp, A. Gy\'{a}rf\'{a}s, Zs. Tuza, Induced matchings in bipartite graphs, Discrete Mathematics 78 (1989) 83-87.
\bibitem{fu} Z. F\"{u}redi, F. Lazebnik, A. Seress, V.A. Ustimenko, A.J. Woldar, Graphs of prescribed girth and bi-degree, Journal of Combinatorial Theory, Series B 64 (1995) 228-239. 
\bibitem{felera} M. F\"{u}rst, M. Leichter, D. Rautenbach, Locally searching for large induced matchings, Theoretical Computer Science 720 (2018) 64-72.
\bibitem{gole} Z. Gotthilf, M. Lewenstein, Tighter approximations for maximum induced matchings in regular graphs, Lecture Notes in Computer Science 3879 (2006) 270-281.
\bibitem{jo} F. Joos, Induced matchings in graphs of bounded degree, SIAM Journal on Discrete Mathematics 30 (2016) 1876-1882.
\bibitem{jorasa} F. Joos, D. Rautenbach, T. Sasse, Induced matchings in subcubic graphs, SIAM Journal on Discrete Mathematics 28 (2014) 468-473.
\bibitem{limeva} M.C. Lin, J. Mestre, S. Vasiliev, Approximating weighted induced matchings, Discrete Applied Mathematics 243 (2018) 304-310.
\bibitem{orfigozv} Y. Orlovich, G. Finke, V. Gordon, I. Zverovich, Approximability results for the maximum and minimum maximal induced matching problems, Discrete Optimization 5 (2008) 584-593.
\bibitem{ra} D. Rautenbach, Two greedy consequences for maximum induced matchings, Theoretical Computer Science 602 (2015) 32-38.
%\bibitem{stva} L.J. Stockmeyer, V.V. Vazirani, NP-completeness of some generalizations of the maximum matching problem, Information Processing Letters 15 (1982) 14-19.
\bibitem{zi} M. Zito, Induced matchings in regular graphs and trees, Lecture Notes in Computer Science 1665 (1999) 89-101.
\end{thebibliography}
\end{document}